\documentclass[a4paper,draft]{amsproc}

\usepackage{amssymb}
\usepackage{amscd} 
\usepackage[dvips]{graphicx}

\theoremstyle{plain}
 \newtheorem{thm}{Theorem}[section]
 \newtheorem{prop}{Proposition}[section]
 
 \newtheorem{cor}{Corollary}[section]
\theoremstyle{definition}

\theoremstyle{remark}
 \newtheorem{rem}{Remark}[section] 
 \numberwithin{equation}{section}

\def\nab{\nabla}
\def\wnab{\widetilde{\nabla}} 
\def\span{\mathop{\rm span}\nolimits} 
\def\wT{\mathop{\widetilde{T}}\nolimits} 
\def\wR{\mathop{\widetilde{R}}\nolimits} 
\def\wK{\mathop{\widetilde{K}}\nolimits} 
\def\wS{\mathop{\widetilde{S}}\nolimits} 
\def\tr{\mathop{tr}\nolimits}
\def\dv{\mathop{div}\nolimits}
\def\wr{\widetilde{r}}

\title[Schouten-van Kampen affine connections]
{THE SCHOUTEN-VAN KAMPEN AFFINE \\
CONNECTION ADAPTED TO AN ALMOST \\
(PARA) CONTACT METRIC STRUCTURE}

\author[Zbigniew Olszak]{\bfseries Zbigniew Olszak}

\date{\today}



\address{Institute of Mathematics and Computer Science \\
Wroc{\l}aw University of Technology \\
Wybrze\.ze Wyspia\'nskiego 27 \\
50-370 Wroc{\l}aw \\
Poland}

\email{zbigniew.olszak@pwr.wroc.pl}

\begin{document}


\vspace{18mm}
\setcounter{page}{1}
\thispagestyle{empty}

\begin{abstract}
We study the Schouten-van Kampen connection associated to an almost contact or paracontact metric structure. With the help of such a connection, some classes of almost (para) contact metric manifolds are characterized. Certain curvature properties of this connection are found. 
\end{abstract}

\maketitle

\section{Introduction}

The Schouten-van Kampen connection is one of the most natural connections adapted to a pair of complementary distributions on a differentiable manifold endowed with an affine connection; cf.\ \cite{BF1,I,SK}, etc. We would like to pay much attention to the papers \cite{Sol0} -- \cite{Sol2} by A.\,F. Solov'ev, who has investigated hyperdistributions in Riemannian manifolds using the Schouten-van Kampen connection. 

On the other hand, any almost contact as well as paracontact metric manifold admits a hyperdistribution. Such distributions and some kinds of affine connections adapted to these distributions were studied by many authors; see \cite{AMOO,Be,BF2,Bl,KiPa,KB,SaTa,Sol3,Sol4}, etc. 

In this short note, we are interested in Schouten-van Kampen connections which are associated to the hyperdistributions occuring on almost contact as well as paracontact (possibly indefinite) metric manifolds. With the help of the Schouten-van Kampen connection, we characterize some classes of almost (para) contact metric manifolds, and find certain curvature properties of this connection on these manifolds. 

\section{Hyperdistributions in pseudo-Riemannian manifolds}

Let $M$ be a (connected) pseudo-Riemannian manifold of an arbitrary signature $(p,n-p)$, $0\leqslant p\leqslant n$, $n=\dim M\geqslant2$. By $g$ will be denoted the pseudo-Riemannian metric on $M$, and by $\nab$ the Levi-Civita connection coming from the metric $g$.

Assume that $\mathcal H$ and $\mathcal V$ are two complementary, orthogonal distributions on $M$ such that $\dim\mathcal H=n-1$, $\dim\mathcal V=1$, and the distribution $\mathcal V$ is non-null. Thus, $T\!M = \mathcal H\oplus\mathcal V$, $\mathcal H\cap\mathcal V = \{0\}$ and $\mathcal H\perp\mathcal V$. Asume that $\xi$ is a unit vector field and $\eta$ is a linear form such that $\eta(\xi)=1$, $g(\xi,\xi)=\varepsilon=\pm1$ and 
\begin{equation}
\label{pair}
  \mathcal H = \ker\eta,\quad
  \mathcal V = \span\{\xi\}. 
\end{equation}
We can always choose such $\xi$ and $\eta$ at least locally (in a certain neighborhood of an arbitrary chosen point of $M$). Then, we also have $\eta(X) = \varepsilon g(X,\xi)$. Moreover, it holds that $\nabla_X\xi\in \mathcal H$. 

For any $X\in T\!M$, by $X^h$ and $X^v$ we denote the projections of $X$ onto $\mathcal H$ and $\mathcal V$, respectively. Thus, we have $X=X^h+X^v$ with 
\begin{equation}
\label{hv}
  X^h=X-\eta(X)\xi,\quad X^v=\eta(X)\xi.
\end{equation}

The Schouten-van Kampen connection $\wnab$ associated to the Levi-Civita connection $\nabla$ and adapted to the pair of distributions $(\mathcal H,\mathcal V)$ is defined by (cf.\ e.g.\ \cite{BF1})
\begin{equation}
\label{SvK0}
  \wnab_XY = (\nab_XY^h)^h + (\nab_XY^v)^v,
\end{equation} 
and the corresponding second fundamental form $B$ is defined by $B = \nab - \wnab$. Note that the condition (\ref{SvK0}) implies the parallelity of the distributions $\mathcal H$ and $\mathcal V$ with respect to the Schouten-van Kampen connection $\wnab$. 

Having (\ref{hv}), one can compute
\begin{eqnarray*}
  (\nab_XY^h)^h &=& \nab_XY - \eta(\nab_XY)\xi - \eta(Y)\nab_X\xi, \\
  (\nab_XY^v)^v &=& ((\nab_X\eta)(Y) + \eta(\nab_XY))\xi,
\end{eqnarray*}
which enables us to express the Schouten-van Kampen connection with the help of the Levi-Civita connection in the following way (cf. \cite{Sol0})
\begin{equation}
\label{SvK}
  \wnab_XY = \nab_XY - \eta(Y)\nab_X\xi + (\nab_X\eta)(Y)\xi. 
\end{equation}
Thus, the second fundamental form $B$ and the torsion $\wT$ of $\wnab$ are (cf. \cite{Sol0} - \cite{Sol00})
\begin{eqnarray}
\label{sff}
  B(X,Y) &=& \eta(Y)\nab_X\xi - (\nab_X\eta)(Y)\xi, \\
\label{tor}
  \wT(X,Y) &=& \eta(X)\nab_Y\xi - \eta(Y)\nab_X\xi + 2d\eta(X,Y)\xi.
\end{eqnarray}
The linear operator $L$ defined by 
\begin{equation}
\label{shape}
  LX = -\nab_X\xi
\end{equation}
will be called the shape operator. It can be noticed that $\xi$ is a Killing vector field if anf only if $L$ is an antisymmetric operator. From (\ref{sff}), we see that $B$ can be described with the help of the shape operator $L$, namely
\begin{equation*}
  B(X,Y) = \null-\eta(Y)LX + \varepsilon g(LX,Y)\xi.
\end{equation*}
Moreover, from (\ref{sff}) and (\ref{tor}), we deduce that $\wT = - 2\mathcal A(B)$, $\mathcal A$ being the antisymmetrization operation.

With the help of the Schouten-van Kampen connection (\ref{SvK}), many properties of some geometric objects connected with the distributions $\mathcal H$, $\mathcal V$ can be characterized (cf.\, \cite{Sol0} - \cite{Sol1}). Probably, the most spectacular is the following statement: $g$, $\xi$ and $\eta$ are parallel with respect to $\wnab$, that is, $\wnab\xi = 0$, $\wnab g = 0$, $\wnab\eta = 0$. 

We finish this section with the following statement: 

\begin{thm}
The Schouten-van Kampen connection $\wnab$ associated to the Levi-Civita connection $\nabla$ and adapted to the pair (\ref{pair}) is just the only one affine connection, which is metric and its torsion is of the form (\ref{tor}).
\end{thm}

\begin{proof}
It remains to prove that if an affine connection $\wnab$ is metric ($\wnab g=0$) and its torsion is given by (\ref{tor}), then it is given by (\ref{SvK}). For, recall the famous result stating that any metric connection can be expressed with the help of its torsion $\wT$ in the following way 
\begin{eqnarray*}
  g(\wnab_XY,Z) 
    &=& g(\nab_XY,Z) \\
    & & \null + \frac12 (g(\wT(X,Y),Z)
                             -g(\wT(X,Z),Y)
                             -g(\wT(Y,Z),X)).
\end{eqnarray*}
Applying (\ref{tor}) into the above relation enables us to deduce the following
\begin{equation*}
  g(\wnab_XY,Z) = g(\nab_XY,Z)
  - \eta(Y)g(\nab_X\xi,Z) + \eta(Z)g(\nab_X\xi,Y),
\end{equation*}
which immediatelly leads to (\ref{SvK}). 
\end{proof}

\section{The curvature of the Schouten-van Kampen connection}

We keep the assumptions and notations from the previous section. 

Let $\wR$ and $R$ be the curvature operators of the Levi-Civita connection $\nab$ and the Schouten-van Kampen connection $\wnab$,  
\begin{equation*}
  \wR(X,Y) = \big[\wnab_X,\wnab_Y\big] - \wnab_{[X,Y]},\quad
  R(X,Y) = \big[\nab_X,\nab_Y\big] - \nab_{[X,Y]}.
\end{equation*}
Using (\ref{SvK}), by direct calculations, we obtain the following formula connecting $\wR$ and $R$ (cf. \cite{Sol0})
\begin{eqnarray}
\label{curvs0}
  \wR(X,Y)Z 
  &=& R(X,Y)Z - \eta(R(X,Y)Z)\xi - \eta(Z) R(X,Y)\xi \\
  & & \null + (\nab_Y\eta)(Z)\nab_X\xi - (\nab_X\eta)(Z)\nab_Y\xi.\nonumber
\end{eqnarray}

We can write the formula (\ref{curvs0}) with the help of the shape operator $L$. To do it, we need to use the following consequences of (\ref{shape}) 
\begin{eqnarray}
\label{xx1}
  (\nab_X\eta)(Y) &=& -\varepsilon g(LX,Y), \\
  R(X,Y)\xi &=& \null- (\nab_X L)Y + (\nab_Y L)X, \nonumber \\
  \eta(R(X,Y)Z) &=& \varepsilon g((\nab_X L)Y - (\nab_Y L)X,Z). \nonumber
\end{eqnarray}
Now, applying the above relations and (\ref{shape}) into (\ref{curvs0}), we obtain the following 

\begin{thm}
The curvature operators $\wR$ and $R$ are related by 
\begin{eqnarray}
\label{curvs}
  \wR(X,Y)Z 
  &=& R(X,Y)Z - \varepsilon g((\nab_X L)Y-(\nab_Y L)X,Z)\xi \\
  & & \null + \eta(Z)((\nab_X L)Y-(\nab_Y L)X) \nonumber \\
  & & \null + \varepsilon g(LY,Z)LX - \varepsilon g(LX,Z)LY.\nonumber
\end{eqnarray}
\end{thm}

We will also consider the Riemann curvature $(0,4)$-tensors $\wR$, $R$, and the Ricci curvature tensors $\wS$, $S$, and the scalar curvatures $\wr$, $r$ of the connections $\wnab$ and $\nab$ defined by 
\begin{eqnarray*}
  & \wR(X,Y,Z,W) = g(\wR(X,Y)Z,W), \quad
    R(X,Y,Z,W) = g(R(X,Y)Z,W), & \\
  & \wS(Y,Z) = \tr\{X\mapsto \wR(X,Y)Z\}, \quad
    S(Y,Z) = \tr\{X\mapsto R(X,Y)Z\}, & \\
  & \wr = \tr_g \{(Y,Z)\mapsto\wS(Y,Z)\}, \quad
    r = \tr_g \{(Y,Z)\mapsto S(Y,Z)\}.
\end{eqnarray*}
Having (\ref{curvs}), we obtain 

\begin{cor}
The Riemann curvature $(0,4)$-tensors $\wR$ and $R$ are related by 
\begin{eqnarray}
\label{curvs2}
  \wR(X,Y,Z,W) 
  &=& R(X,Y,Z,W) - g((\nab_X L)Y-(\nab_Y L)X,Z)\eta(W) \\
  & & \null + g((\nab_X L)Y-(\nab_Y L)X,W)\eta(Z) \nonumber \\
  & & \null + \varepsilon g(LX,W)g(LY,Z) 
            - \varepsilon g(LX,Z)g(LY,W).\nonumber
\end{eqnarray}
\end{cor}

Note that from (\ref{curvs2}), the skew-symmetry of $\wR$ with respect to the last two arguments follows additionally, that is, $\wR(X,Y,Z,W) = - \wR(X,Y,W,Z)$. 

It is also worthwhile to notice that $\eta(LY) = \varepsilon g(LY,\xi) = 0$ by  (\ref{shape}). Hence, for the covariant derivative $\nab L$, we deduce
\begin{equation*}
  g((\nab_XL)Y,\xi) = g(LX,LY).
\end{equation*}

Some additional consequences of (\ref{curvs0}) can be stated as it follows. 

\begin{cor}
The Ricci curvature tensors $\wS$ and $S$ are related by 
\begin{eqnarray*}
  \wS(Y,Z) 
  &=& S(Y,Z) - \varepsilon R(\xi,Y,Z,\xi) - S(\xi,Y)\eta(Z) \\
  & & \null+ (\nab_Y\eta)(Z)\dv\xi - (\nab_{\nab_Y\xi}\eta)(Z), 
\end{eqnarray*}
where $\dv\xi = \tr\{X\mapsto\nab_X\xi\}$ is the divergence of the vector field $\xi$. 
\end{cor}

\begin{cor}
The scalar curvatures $\wr$ and $r$ are related by 
\begin{equation*}
  \wr = 
  r - 2\varepsilon S(\xi,\xi) + \varepsilon (\dv\xi)^2 - \tr_g\{(Y,Z)\mapsto\big(\nab_{\nab_Y\xi}\eta)(Z)\big).
\end{equation*}
\end{cor}

As usually, a non-degenerate section $\sigma$ is an arbitrary 2-dimensional subspace of a tangent space $T_pM$, $p\in M$, such that $g|_{\sigma}$ is of algebraic rank 2. The sectional curvatures for $\nab$ and $\wnab$ are defined in the standard way
\begin{eqnarray*}
  K(\sigma) &=& R(X,Y,Y,X)\cdot(g(X,X)g(Y,Y) - g^2(X,Y))^{-1}, \\
  \wK(\sigma) &=& \wR(X,Y,Y,X)\cdot(g(X,X)g(Y,Y) - g^2(X,Y))^{-1},
\end{eqnarray*}
where the pair $X,Y$ is a basis in $\sigma$. 

\begin{cor}
The sectional curvatures $\wK$ and $K$ are related by 
\begin{eqnarray*}
  \wK(\sigma) &=& 
  K(\sigma) + (g(X,X)g(Y,Y) - g^2(X,Y))^{-1} \cdot \\
  & & \cdot \big(\null- \eta(X) R(X,Y,Y,\xi) + \eta(Y) R(Y,X,X,\xi) \nonumber\\
  & & \ \null+ \varepsilon (\nab_X\eta)(X) (\nab_Y\eta)(Y) - \varepsilon (\nab_X\eta)(Y) (\nab_Y\eta)(X)\big), \nonumber
\end{eqnarray*}
where the pair $X,Y$ is a basis of a non-degenerate section $\sigma$. 
\end{cor}

\section{Almost (para) contact metric manifolds}

In the geometric literature, we find various classes of almost contact or paracontact metric structures; see e.g.\ \cite{Bl,BCGRH,BG,E,KW,TT,Z}, etc. The following convention unifying both the contact and paracontact notations seems to be useful for our purposes. It is a small generalization of the idea applied by S.\ Erdem in \cite{E}.

Let $M$ be a $(2n+1)$-dimensional (connected) differentiable manifold endowed with a quadruplet $(\varphi,\xi,\eta,g)$, where $\varphi$ is $(1,1)$-tensor field, $\xi$ is a vector field, $\eta$ is a 1-form, and $g$ is a pseudo-Riemannian such that
\begin{eqnarray*}
  & \varphi^2X = \mu(X - \eta(X)\xi),\quad 
     \eta(\xi)=1, \\
  & g(\varphi X,\varphi Y)=-\mu(g(X,Y)-\varepsilon\eta(X)\eta(Y)), &
\end{eqnarray*}
where $\varepsilon,\mu=\pm1$. As a consequence of the above conditions, we have additionally 
\begin{equation*}
  \varphi\xi=0,\quad 
  \eta\circ\varphi=0,\quad 
  \eta(X)=\varepsilon g(X,\xi),\quad 
  g(\xi,\xi)=\varepsilon. 
\end{equation*}
The manifold $M$ will be called almost (para) contact metric, and the quadruplet $(\varphi,\xi,\eta,g)$ will be called the almost (para) contact metric structure on $M$. For such a manifold, the fundamental 2-form $\varPhi$ (a skew-symmetric $(0,2)$-tensor field of maximal algebraic rank ($=2n$)) is defined by $\varPhi(X,Y)=g(X,\varphi Y)$. 

When $\mu=-1$, then the manifold $M$ is an almost contact metric manifold. In this case the metric $g$ is assumed to be pseudo-Riemannian in general, including Riemannian. Thus, if $\varepsilon=1$, the signature of $g$ is equal to $2p$, where $0\leqslant p\leqslant n$; and if $\varepsilon=-1$, the signature of $g$ is equal to $2p+1$, where $0\leqslant p\leqslant n$.

When $\mu=1$, then the manifold $M$ is an almost paracontact metric manifold. In this case, the metric $g$ is pseudo-Riemannian, and its signature is equal to $n$ when $\varepsilon=1$, or $n+1$ when $\varepsilon=-1$. One notes that in this case, the eigenspaces of the linear operator $\varphi$ corresponding to the eigenvalues 1 and $-1$ are both $n$-dimensional at every point of the manifold.

\section{The Schouten-van Kampen connection adapted to an almost (para) contact metric structure}

\subsection{Certain general conclusions} 
Let $M$ an almost (para) contact metric manifold, and consider the following pair of complementary and orthogonal distributions
\begin{equation}
\label{distHV}
  \mathcal H = \ker\eta,\quad
  \mathcal V = \mathop{\rm span}\{\xi\}.
\end{equation}
We have $\dim\mathcal H=2n$ and $\dim\mathcal V=1$. $\mathcal H$ is usually called the contact or paracontact or canonical distributtion. We will call it a (para) contact distribution. 

The Schouten-van Kampen connection adapted to the pair (\ref{distHV}) and arising from the Levi-Civita connection $\nab$ will be called the Schouten-van Kampen connection adapted to the almost (para) contact metric structure, and will be denoted by $\wnab$. As we have seen in Section 1, the connection $\wnab$ is given by (\ref{SvK}). 

Now, by a direct calculation in which (\ref{SvK}) should be used, we obtain the main formula for $\wnab\varphi$. 

\begin{prop}
\label{main}
For an almost (para) contact metric manifold, we have 
\begin{equation}
\label{nabphi}
  (\wnab_X\varphi)Y = (\nab_X\varphi)Y 
  + \eta(Y)\varphi\nab_X\xi - \varepsilon g(\varphi \nab_X\xi,Y)\xi.
\end{equation}
\end{prop}

As we can see from the above proposition, the condition 
\begin{equation}
\label{import}
  (\nab_X\varphi)Y 
  = \varepsilon g(\varphi \nab_X\xi,Y)\xi - \eta(Y)\varphi\nab_X\xi 
\end{equation}
is very important since it just means that $\wnab\varphi=0$. The condition (\ref{import}) will be used many times in the rest of the paper. The case $\dim M=3$ is the first situation where it occurs.

\begin{prop}
\label{pro3d}
For a 3-dimensional almost (para) contact metric manifold, the condition (\ref{import}) is satisfied. Consequently, for a such a manifold, we have \linebreak $\wnab\varphi = 0$.
\end{prop}

\begin{proof}
The idea of the proof of the first assertion is precisely the same as that of \cite[Proposition 1]{Ol} and \cite[Proposition 2.2]{W}. Thus, we omit it. The second assertion follows now (\ref{import}) and Proposition \ref{main}.
\end{proof}

\subsection{(Para) $\alpha$-contact metric manifolds}
An almost (para) contact metric manifold will be called (para) $\alpha$-contact if $d\eta=\alpha\varPhi$ for a certain non-zero function $\alpha$; and (para) $K$-$\alpha$-contact if it is (para) $\alpha$-contact and $\xi$ is additionally a Killing vector field; cf. \cite{KR}. In the case when $\alpha=1$, we have a (para) contact metric manifold and a (para) $K$-contact manifold, respectively. 

\begin{prop}
\label{cont}
An almost (para) contact metric manifold is \\
(a) (para) $\alpha$-contact if and only if $L-\varepsilon\alpha\varphi$ is a symmetric linear operator; \\
(b) (para) $K$-$\alpha$-contact if and only if $L=\varepsilon\alpha\varphi$, \\
where in the both above cases, $\alpha$ is a certain non-zero function.
\end{prop}

\begin{proof}
(a) Note that by (\ref{xx1}), we have 
\begin{eqnarray*}
  && g((L-\varepsilon\alpha\varphi)X,Y)
  - g((L-\varepsilon\alpha\varphi)Y,X) \\
  && \qquad\qquad = \varepsilon(2\alpha\varPhi(X,Y) - (\nab_X\eta)(Y) + (\nab_Y\eta)(X))\\
  && \qquad\qquad = 2\varepsilon(\alpha\varPhi(X,Y) - d\eta(X,Y)).
\end{eqnarray*}
Thus, $d\eta=\alpha\varPhi$ if and only if $L-\varepsilon\alpha\varphi$ is a symmetric linear operator.

(b) Note that $\xi$ is Killing if and only if $L$ is a skew-symmetric linear operator, or equivalently $L-\varepsilon\alpha\varphi$ is a skew-symmetric linear operator. This constatation together with (a) gives our assertion (b). 
\end{proof}

Many curvature properties of the Schouten-van Kampen connections on contact or $K$-contact manifolds ($\varepsilon=1$, $\mu=-1$) with positive definite metric were achived in \cite{Sol3,Sol4}. 

\subsection{Normal almost (para) contact metric manifolds}
An almost (para) contact metric manifold (structure) will be called normal (\cite{Bl,KW}) if the almost (para) complex structure $J$ defined on $M\times\mathbb R$ by 
\begin{equation*}
  J\Big(X,a\frac{\partial}{\partial t}\Big) = 
  \Big(\varphi X + \mu a \xi,
  \eta(X)\frac{\partial}{\partial t}\Big)
\end{equation*}
is integrable, or equivalently 
\begin{equation*}
  [\varphi,\varphi](X,Y) - 2\mu d\eta(X,Y)\xi=0,
\end{equation*}
$[\varphi,\varphi]$ being the Nijehuis torsion tensor of $\varphi$, defined by 
\begin{equation*}
  [\varphi,\varphi](X,Y) = [\varphi X,\varphi Y]-\varphi [X,\varphi Y]-\varphi [\varphi X,Y]+\varphi^2[X,Y].
\end{equation*}
Note that if $\mu=-1$, then $J$ is an almost complex structure, and if $\mu=1$, then $J$ is an almost paracomplex structure. 

\begin{prop}
\label{proarb}
An almost (para) contact metric manifold is normal if and only if the shape operator $L$ commutes with $\varphi$ and 
\begin{equation}
\label{newnormcond}
  (\wnab_{\varphi X}\varphi)\varphi Y + \mu (\wnab_X\varphi)Y=0. 
\end{equation} 
\end{prop}

\begin{proof}
Recalling \cite[Lemma, p. 171]{ST} and \cite[Proposition 2.1]{W}, we claim that the normality condition of an almost (para) contact structure can be formulated with the help of $\nab\varphi$ in the following way  
\begin{equation}
\label{norma}
  (\nab_{\varphi X}\varphi)Y - \varphi(\nab_X\varphi)Y
  -\varepsilon\mu g(\nab_X\xi,Y)\xi=0, 
\end{equation}
or equivalently 
\begin{equation}
\label{normcond}
  (\nab_{\varphi X}\varphi)\varphi Y + \mu (\nab_X\varphi)Y 
   + \mu\eta(Y)\varphi\nab_X\xi=0. 
\end{equation}
Note also that the normality condtion always implies 
\begin{equation}
\label{La}
  \nab_{\varphi X}\xi = \varphi\nab_X\xi. 
\end{equation}
In fact, (\ref{La}) follows easily from (\ref{norma}) when we put there $Y=\xi$. By (\ref{shape}), the relation (\ref{La}) is equivalent to the commutativity of $L$ and $\varphi$, that is $L\varphi=\varphi L$. 

Before we finish the proof, using (\ref{nabphi}), we find the following general formula for an arbitrary almost (para) contact metric manifold, 
\begin{eqnarray}
\label{gener}
  && (\wnab_{\varphi X}\varphi)\varphi Y + \mu (\wnab_X\varphi)Y
     = (\nab_{\varphi X}\varphi)\varphi Y + \mu (\nab_X\varphi)Y \\
  && \qquad\qquad\null + \mu\eta(Y)\varphi\nab_X\xi 
     + \varepsilon\mu 
     g(\nab_{\varphi X}\xi-\varphi\nab_X\xi,Y)\xi. \nonumber
\end{eqnarray}

If our almost (para) contact metric structure is normal, then applying (\ref{normcond}) and (\ref{La}) into (\ref{gener}), we obtain (\ref{newnormcond}). 

If (\ref{newnormcond}) and (\ref{La}) hold, then from (\ref{gener}) we deduce (\ref{normcond}), which gives the normality. 
\end{proof}

\begin{prop}
For a $3$-dimensional almost (para) contact metric manifold, the following conditions are equivalent: \\
\indent(a) the manifold is normal, \\
\indent(b) the shape operator $L$ commutes with $\varphi$, \\
\indent(c) the shape operator is given by 
\begin{equation}
\label{ncon}
  LX = \varepsilon\alpha\varphi X - \beta(X - \eta(X)\xi), 
\end{equation}
$\alpha$ and $\beta$ being certain functions on $M$.
\end{prop}

\begin{proof}
As we already know (see Proposition \ref{pro3d}), $\wnab\varphi=0$ for a 3-dimensional almost (para) contact metric manifold. Therefore, the equivalence (a) $\Leftrightarrow$ (b) follows from Proposition \ref{proarb}. It is obvious that (c) $\Rightarrow$ (b). Finally, the implication (b) $\Rightarrow$ (c) can be easily verified when we use an adapted $\varphi$-basis $(e_1,e_2=\varphi e_1,e_3=\xi)$. 
\end{proof}

There is an additional differential equation related to the functions $\alpha$ and $\beta$ for an arbitrary 3-dimensional normal almost (para) contact manifold. To get it, using (\ref{ncon}), we obtain 
\begin{equation}
\label{ncon2}
  \nab_X\xi = \null- \varepsilon\alpha\varphi X + \beta(X-\eta(X)\xi). 
\end{equation}
Hence, 
\begin{equation}
\label{ncon3}
  (\nab_X\eta)(Y) = \null- \alpha g(\varphi X,Y)
      + \beta(\varepsilon g(X,Y) - \eta(X)\eta(Y)).
\end{equation}
Moreover, using (\ref{ncon2}), from (\ref{import}), we find 
\begin{equation}
\label{ncon4}
  (\nab_X\varphi)Y 
  = \null-\mu\alpha\big(g(X,Y)\xi - \varepsilon\eta(Y)X\big) 
    + \beta\big(\varepsilon g(\varphi X,Y)\xi - \eta(Y)\varphi X\big).
\end{equation}
Using (\ref{ncon3}) and (\ref{ncon4}), for the exterior derivatives of $\eta$ and $\varPhi$, we get $d\eta=\alpha\varPhi$ and $d\varPhi = 2\beta\eta\wedge\varPhi$. Therefore, $0 = d^2\eta = (d\alpha + 2\alpha\beta\eta)\wedge\varPhi$. Since $\dim M=3$, from the last equality, the following interesting equation follows 
\begin{equation*}
  d\alpha(\xi) + 2\alpha\beta = 0.
\end{equation*}

In the next sections, we will study curvature properties of some subclasses of the class of normal almost (para) contact metric manifolds. 

\subsection{(Para) $\alpha$-Sasakian manifolds}

We extend the notion of $\alpha$-Sasakian \linebreak manifolds (see e.g.\ \cite{JV,Bl,KR}), and call an almost (para) contact metric manifold to be (para) $\alpha$-Sasakian if it satisfies the condition 
\begin{equation}
\label{sasa}
  (\nab_X\varphi)Y 
  = \null-\mu\alpha\big(g(X,Y)\xi-\varepsilon\eta(Y)X\big), 
\end{equation}
$\alpha$ being a function. Similarly as for $\alpha$-Sasakian manifolds, it can be proved that an almost (para) contact metric manifold is (para) $\alpha$-Sasakian if and only if it is normal and (para) $\alpha$-contact. As a consequence of (\ref{sasa}), one obtains also $d\varPhi=0$. Therefore, $0=d^2\eta=d\alpha\wedge\varPhi$. Consequently, in dimensions $2n+1\geqslant5$, it must be that $d\alpha=0$, that is, $\alpha$ is constant.

\begin{prop}
An almost (para) contact metric manifold is (para) $\alpha$-Sasakian if and only if it is (para) $K$-$\alpha$-contact and $\wnab\varphi=0$. 
\end{prop}

\begin{proof}
It is a straighforward verification that the condtion (\ref{sasa}) is fulfilled if and only if the condtions (\ref{import}) and 
\begin{equation}
\label{sasak2}
  \nab_X\xi = \null-\varepsilon\alpha\varphi X
\end{equation}
hold simultanously. By virtue of (\ref{nabphi}), the condition (\ref{import}) is equivalent to $\wnab\varphi=0$. And, the condition (\ref{sasak2}) means that the manifold is (para) $K$-$\alpha$-contact; see Proposition \ref{cont}. 
\end{proof}

Using the formula (\ref{curvs2}), we describe the relations between the curvatures of the Levi-Civita and the Schouten-van Kampen connections for (para) $\alpha$-Sasakian manifolds in dimensions $2n+1\geqslant5$. But at first, using $LY = \varepsilon\alpha\varphi Y$ (which follows from (\ref{sasak2})), and  (\ref{shape}), (\ref{xx1}), (\ref{sasa}), we find
\begin{equation*}
  (\nab_X L)Y - (\nab_Y L)X = \mu\alpha(\eta(Y)X - \eta(X)Y).
\end{equation*} 
Having the above in mind, from (\ref{curvs2}), we obtain the following:

\begin{thm} 
For an (para) $\alpha$-Sasakian manifold of dimension $2n+1\geqslant5$, the Riemann curvatures $\wR$, $R$, the Ricci curvatures $\wS$, $S$, and the scalar curvatures $\wr$, $r$ are related by the following formulas 
\begin{eqnarray*}
  \wR(X,Y,Z,W) &=& R(X,Y,Z,W) + \mu\alpha
      \big(g(X,W)\eta(Y)-g(Y,W)\eta(X)\big)\eta(Z) \\
   & & \null + \mu\alpha\big(g(Y,Z)\eta(X)-g(X,Z)\eta(Y)\big)\eta(W) \\
   & & \null + \varepsilon\alpha^2\big(g(X,\varphi W)g(Y,\varphi Z) 
             - g(X,\varphi Z)g(Y,\varphi W)\big), \\
  \wS(Y,Z) &=& S(Y,Z) + \varepsilon\mu\alpha(1-\alpha)g(Y,Z)
             + \mu\alpha(2n-1+\alpha)\eta(Y)\eta(Z), \\
  \wr &=& r + 2n\varepsilon\mu\alpha(2-\alpha).
\end{eqnarray*}
\end{thm}

\begin{cor}
For a (para) $\alpha$-Sasakian manifold of dimension $2n+1\geqslant5$, the sectional curvatures curvatures $\wK$, $K$ of a nondegenerate section $\sigma$ are related by the formulas 
\begin{eqnarray*}
  \wK(\sigma) &=& K(\sigma) + \varepsilon\alpha^2\quad
                  \mbox{when $\sigma$ is a $\varphi$-section}, \\
  \wK(\sigma) &=& K(\sigma) + \varepsilon\mu\alpha\quad
                  \mbox{when}\quad \xi\in\sigma.
\end{eqnarray*}
\end{cor}

\subsection{(Para) $\beta$-Kemotsu manifolds}

Extending the notion of $\beta$-Kenmotsu manifolds (cf.\ \cite{JV,OR,CC}, etc.), we define an almost (para) contact metric manifold to be (para) $\beta$-Kenmotsu if 
\begin{equation}
\label{kenmot}
  (\nab_X\varphi)Y 
   = \beta\big(\varepsilon g(\varphi X,Y)\xi - \eta(Y)\varphi X\big), 
\end{equation}
$\beta$ being a function on $M$. Similar as for $\beta$-Kenmotsu manifolds, it can be proved that an almost (para) contact metric manifold is (para) $\beta$-Kenmotsu if and only if it is normal and 
\begin{equation}
\label{yy}
  d\varPhi=2\beta\eta\wedge\varPhi,\quad d\eta=0. 
\end{equation}
Note that (\ref{yy}) implies $0=d^2\varPhi=2d\beta\wedge\eta\wedge\varPhi$. Hence, in dimensions $2n+1\geqslant5$, we have $\beta\wedge\eta=0$, by pure algebraic reasons. Coensequently, $d\beta=d\beta(\xi)\eta$. Denoting $\beta' = d\beta(\xi) = \xi(\beta)$, we will write $d\beta = \beta'\eta$. This is a strong restriction for the function $\beta$ in those dimensions. 

\begin{prop}
An almost (para) contact metric manifold is (para) \linebreak $\beta$-Kenmotsu if and only if $\wnab\varphi = 0$ and 
\begin{equation}
\label{kenm}
  L = \beta(\null-I+\xi\otimes\eta).
\end{equation}
\end{prop}

\begin{proof}
It is a straighforward verification that the condtion (\ref{kenmot}) is fulfilled if and only if the two condtions (\ref{import}) and 
\begin{equation}
\label{kenmot2}
  \nab_X\xi = \beta(X - \eta(X)\xi)
\end{equation}
hold simultanously. By (\ref{nabphi}), the condition (\ref{import}) is equivalent to $\wnab\varphi=0$. And, by (\ref{shape}), the condition (\ref{kenmot2}) is equivalent to (\ref{kenm}).
\end{proof}

We describe the relations between the curvatures of the Levi-Civita and the Schouten-van Kampen connections for a (para) Kenmotsu manifold in dimensions $2n+1\geqslant5$. 

As previously, we use the general formula (\ref{curvs2}). But at first, using (\ref{kenm}), (\ref{shape}) and (\ref{xx1}), we find 
\begin{equation*}
  (\nab_X L)Y - (\nab_Y L)X 
  = \null- (\beta'+\beta^2)(\eta(X)Y-\eta(Y)X).
\end{equation*}
Having the above in mind, from (\ref{curvs2}), we obtain the following:

\begin{thm}
For a (para) $\beta$-Kenmotsu manifold of dimension $2n+1\geqslant5$, the Riemann curvatures $\wR$, $R$, the Ricci curvatures $\wS$, $S$, and the scalar curvatures $\wr$, $r$ are related by the formulas
\begin{eqnarray*}
  \wR(X,Y,Z,W) 
   &=& R(X,Y,Z,W) + \varepsilon\beta^2(g(X,W)g(Y,Z) - g(X,Z)g(Y,W)) \\
   & & \null+ \beta'\big(\eta(X)\eta(W)g(Y,Z) - \eta(X)\eta(Z)g(Y,W) \\
   & & \null- \eta(Y)\eta(W)g(X,Z) + \eta(Y)\eta(Z)g(X,W)\big), \\
  \wS(Y,Z) 
   &=& S(Y,Z) + \varepsilon(\beta'+2n\beta^2)g(Y,Z)
       + (2n-1)\beta'\eta(Y)\eta(Z), \\
  \wr &=& r + 2n(2n+1)\varepsilon\beta^2 + 4n\varepsilon\beta'. 
\end{eqnarray*}
\end{thm}

\begin{cor}
For a (para) $\beta$-Kenmotsu manifold of dimension $2n+1\geqslant5$, the sectional curvatures curvatures $\wK$, $K$ of a nondegenerate section $\sigma$ are related by the formulas 
\begin{eqnarray*}
  \wK(\sigma) &=& K(\sigma) + \varepsilon\beta^2
                  \quad\mbox{when}\quad \sigma\perp\xi, \\
  \wK(\sigma) &=& K(\sigma) + \varepsilon(\beta'+\beta^2)
                  \quad\mbox{when}\quad \xi\in\sigma.
\end{eqnarray*}
\end{cor}

\subsection{(Para) trans-Sasakian manifolds}

Consider a special subclass of almost (para) contact metric manifolds. Namely, those which satisfy the condtion
\begin{equation}
\label{trans}
  (\nab_X\varphi)Y 
  = \null-\mu\alpha\big(g(X,Y)\xi - \varepsilon\eta(Y)X\big) 
    + \beta\big(\varepsilon g(\varphi X,Y)\xi - \eta(Y)\varphi X\big), 
\end{equation}
where $\alpha$ and $\beta$ are certain functions on $M$. Let us call such manifolds to be (para) trans-Sasakian. Similarly as for trans-Sasakian manifolds, it can proved that an almost (para) contact metric manifold is (para) trans-Sasakian if and only if it is normal and 
\begin{equation}
\label{trrr}
  d\varPhi=2\beta\eta\wedge\varPhi, 
  \quad d\eta=\alpha\varPhi.
\end{equation}

\begin{rem}
The above class of manifolds seems to be a natural generalization of the class of trans-Sasakian manifolds defined in \cite{Ou}, and since then, studied in many papers. It is important that in dimensions $\geqslant5$, the class of trans-Sasakian manifolds splits into two subclasses: $\alpha$-Sasakian manifolds and $\beta$-Kenmostu manifolds, and contrary to that, in dimension $3$, we do not have such a splitting; see \cite{M}. Moreover, it is worth to notice that from \cite[Propositions 1 and 2]{Ol} it can be easily deduced the following (one has only to change the role of the functions $\alpha$ and $\beta$): In dimension 3, the class of trans-Sasakian manifolds is precisely the class of normal almost contact metric manifolds. This fact was also mentioned in \cite{ASKT}.
\end{rem}

The following proposition is a generalization of the facts known for trans-Sasakian manifolds.

\begin{prop}
\ \\
\indent(a) In dimension 3, the class of (para) trans-Sasakian manifolds coincides with the class of normal almost (para) contact metric manifolds. \\
\indent(b) In dimensions $\geqslant5$, the class of (para) trans-Sasakian manifolds splits into two subclasses: (para) $\alpha$-Sasakian manifolds and (para) $\beta$-Kenmostu manifolds. The common part of these subclasses form the (para) cosymplectic manifolds (\;that is, those for which $\nab\varphi=0$). 
\end{prop}

\begin{proof}
(a) Let us assume that the dimension is equal to 3. As we already know, in this dimension, the relation (\ref{import}) is fulfilled. Therefore, (\ref{trans}) holds if and only if the condition (\ref{ncon2}) is satisfied. This condition is the same as (\ref{ncon}), which is equivalent to the normality of our structure. 

(b) At first, the exterior differentiation of both of the relations (\ref{trrr}) gives
\begin{eqnarray}
\label{war1}
   0 &=& d^2\varPhi = 2d\beta\wedge\eta\wedge\varPhi 
         + 2\alpha\beta\varPhi\wedge\varPhi, \\
\label{war2}
   0 &=& d^2\eta = (d\alpha+2\alpha\beta\eta)\wedge\varPhi.
\end{eqnarray}

Let us assume that the dimension is $\geqslant5$. From (\ref{war1}), it follows that $\alpha\beta=0$. Therefore, from (\ref{war2}), we obtain $d\alpha=0$, that is, $\alpha$ is constant. Thus, we have got the first assertion.
\end{proof}


\bibliographystyle{amsplain}

\end{document}